\documentclass[12pt]{amsart}
\usepackage{amsmath}
\usepackage{amsfonts}
\usepackage{latexsym}
\usepackage{amssymb}
\usepackage[dvips]{graphics}
\newcommand{\Z}{{\mathbf Z}}

\newcommand{\R}{{\mathbf R}}

\newcommand{\coker}{{\mbox{\rm coker}}}

\newcommand{\Hom}{{\rm Hom}}

\newcommand{\supp}{\rm {supp}}
\newcommand{\rk}{{\rm {rk}}}

\newcommand{\Lk}{{\rm {Lk}}}

\def\ker{{\rm Ker }}

\def\R{\mathbf R}

\newtheorem{theorem}{Theorem}%[section]

\newtheorem{proposition}[theorem]{Proposition}
\newtheorem{lemma}[theorem]{Lemma}

\newtheorem{corollary}[theorem]{Corollary}
\newtheorem{definition}{Definition}

\newenvironment{example}{\par{\bf Example.}}%{\par\noindent}

\begin{document}
\title[Two particles on a graph]{Topology of configuration space of two particles on a graph, II}
\author{M. ~Farber and E. ~Hanbury}
\address{Department of Mathematics, University of Durham, Durham DH1 3LE, UK}
\email{Michael.Farber@durham.ac.uk}
\email{Elizabeth.Hanbury@durham.ac.uk}
\date{September 25, 2010}

%\thanks{M. Farber was partially supported by a grant from the Israel Academy of Sciences
%and Humanities}

\begin{abstract}
This paper continues the investigation of the configuration space of two distinct points on a graph.
We analyze the process of adding an additional edge to the graph and the resulting
changes in the topology of the configuration space. We introduce a linking bilinear form on the homology group of the graph with values in the cokernel of the intersection form
(introduced in Part I of this work).
For a large class of graphs, which we call \emph{mature} graphs, we give explicit expressions for the homology groups of the configuration space. We show that under a simple condition, adding an edge to a mature graph yields another mature graph.

%MSC: 55R80; 57M15
\end{abstract}

\maketitle
\section{Introduction}

Denote by $F(X, n)$ the space of configurations of $n$ distinct points lying in a topological space $X$.
The configuration spaces $F(X, n)$, first introduced by Fadell and Neuwirth in \cite{FN}, play an important role in modern topology and its applications. The topology of
$F(X, n)$ under various assumptions on $X$ was studied in \cite{Ar}, \cite{Co}, \cite{Va}, \cite{To}.

Recently, important progress in the analysis of the topology of configuration spaces of graphs was made in the work of A. Abrams \cite{Ab} and D. Farley and L. Sabalka \cite{FS1}, \cite{FS2}, \cite{FS3}, \cite{FS4}. The cohomology algebras of unordered configuration spaces of trees were computed; the case of two point configuration spaces of trees was studied in \cite{Far}.

In this paper, which continues \cite{BF}, we study the special case $F(\Gamma, 2)$ where $\Gamma$ is a finite graph. The spaces $F(\Gamma, 2)$ appear in topological robotics as configuration spaces of two objects moving along a one-dimensional network without collisions, see \cite{Gr}, \cite{GK}, \cite{Far}, \cite{Far1}. The space $$F(X, 2)=X\times X - \Delta_X$$ is also known under the name of \lq\lq deleted product\rq\rq; the deleted products of graphs were studied in \cite{Cop}, \cite{CP}, \cite{Patty61} and \cite{Patty62}. The reader should be warned that many statements made by A.H. Copeland and C. W. Patty in
\cite{Cop}, \cite{CP}, \cite{Patty61}, \cite{Patty62}
 are incorrect.

In the first part of this work \cite{BF} the main emphasis was on planar graphs; in this paper we focus predominantly on graphs which are non-planar.

The symbol $H_\ast(X)$ denotes homology groups with integral coefficients.

Let $\Gamma$ be a connected finite graph\footnote{In this paper the term graph means a 1-dimensional simplicial complex.}. Consider the inclusion $\alpha: F(\Gamma, 2) \to \Gamma\times \Gamma$ and the induced homomorphism
$$\alpha_\ast: H_1(F(\Gamma, 2)) \to H_1(\Gamma\times \Gamma).$$ We know that $\alpha_\ast$ is an epimorphism if $\Gamma$ is not homeomorphic to the circle, see Proposition 1.3 from \cite{BF}.
\begin{definition}\label{defmature}
We say that a finite connected graph $\Gamma$ which is not homeomorphic to the interval $[0,1]$ is mature if the homomorphism $\alpha_\ast $ is an isomorphism.
\end{definition}
The term \lq\lq mature\rq\rq\, intends to emphasize that this property is common to all \lq\lq large, well-developed\rq\rq\, graphs. The results presented in this paper justify this intuitive statement.

No planar graph can be mature, see \cite{BF}, Corollary 7.2. A mature graph cannot have vertices of valence one (as follows from Theorem \ref{enlarge1}  below).
The two Kuratowski graphs $K_5$ and $K_{3,3}$ are mature as shown in \cite{BF}, \S 4. The property of a graph to be mature is a topological property,
i.e. it is invariant under subdivisions of the graph.

For a mature graph $\Gamma$ one has
\begin{eqnarray}\label{dimone}
b_1(F(\Gamma, 2)) = 2b_1(\Gamma)
\end{eqnarray}
and the second Betti number $b_2(F(\Gamma, 2))$ equals
\begin{eqnarray}\label{dimtwo}
b_1(\Gamma)^2-b_1(\Gamma) +1-\sum_{v\in V(\Gamma)}(\mu(v)-1)(\mu(v)-2),
\end{eqnarray}
see \cite{BF}, \S 4. Here $V(\Gamma)$ denotes the set of vertices in $\Gamma$ and $\mu(v)$ is the valence of a vertex $v \in V(\Gamma)$. Furthermore, when $\Gamma$ is mature, 
$H_1(F(\Gamma,2))$ and $H_2(F(\Gamma,2))$ are free abelian, see Proposition \ref{prop1}. Thus, we completely describe the homology groups of the configuration space 
$F(\Gamma,2)$ for any mature graph $\Gamma$.

The following theorem illustrates the results of this paper:

\begin{theorem}\label{thm0} Let $\Gamma$ be mature and let $\hat \Gamma=\Gamma\cup e$ be obtained from $\Gamma$ by adding an edge connecting two vertices $u, v\in \Gamma$.
If the complement $\Gamma-\{u, v\}$ is connected then $\hat \Gamma$ is mature as well.
\end{theorem}

Applying this theorem inductively, one may find examples of many mature graphs. In particular we show that complete graphs $K_n$ and bipartite graphs $K_{p,q}$ are mature assuming that $n \ge 5$ and $p\ge3$, $q\ge 3$; this fact was also established by K. Barnett (unpublished) by a different method.

Comparing the results of \cite{BF} and the present paper we see that there is a trichotomy reflecting properties of $F(\Gamma, 2)$ for various classes of graphs: (1) in the case of trees the homomorphism $\alpha_\ast:H_1(F(\Gamma, 2))\to H_1(\Gamma \times \Gamma)$ has a large kernel; (2) if $\Gamma$ is planar, all its vertices have valence $\ge 3$ (and some other technical conditions are satisfied, see Corollary 7.4 in \cite{BF}) then $\alpha_\ast$ has kernel $\Z$; (3) for mature graphs $\alpha_\ast$ is an isomorphism.

The paper is organized as follows. In section 2 we recall the intersection form introduced in \cite{BF} and the formulae for the Betti numbers of $F(\Gamma,2)$ given in terms of this form. In section 3 we look at how the topology of $F(\Gamma,2)$ is changed under two elementary operations on graphs. In section 4 we introduce the \emph{linking homomorphisms} and in sections 5 and 6 we use these to describe what happens to $F(\Gamma,2)$ when we attach an extra edge to $\Gamma$. In section 7 we use these results to investigate necessary and sufficient conditions for a graph to be mature and we give methods for constructing mature graphs. In the final section, section 8, we mention some open questions and conjectures.

We would like to thank the referee for reading the paper very carefully and making many useful comments.

\section{The intersection form}

In this section we recall a construction from \cite{BF}.

Let $\Gamma$ be a connected finite graph. For $x\in\Gamma$ the notation $\supp\{x\}$ stands for the closure of the cell containing $x$.
%The symbol $F(\Gamma, 2)$ denotes the configuration space of two points in $\Gamma$, i.e. $F(\Gamma, 2)=\{(x, y)\in \Gamma\times \Gamma; \, x\not=y\}$.
The subset $D(\Gamma, 2)\subset F(\Gamma, 2)$ is known as the discrete configuration space; it consists of all pairs $(x, y)\in \Gamma\times \Gamma$ with $\supp\{x\}\cap \supp\{y\}=\emptyset$. It is well-known that $F(\Gamma, 2)$ deformation retracts onto $D(\Gamma, 2)$.

Paper \cite{BF} introduced an
intersection form
\begin{eqnarray}\label{intersection}
I=I_\Gamma: H_1(\Gamma) \otimes H_1(\Gamma) \to H_2(N_\Gamma, \partial N_\Gamma)
\end{eqnarray}
which helps to study the homology of the configuration space $F(\Gamma, 2)$.
Here $N_\Gamma$ denotes the neighborhood of the diagonal $\Gamma \subset \Gamma\times \Gamma$ defined as the set of all pairs $(x,y)\in \Gamma\times \Gamma$
such that $x$ and $y$ admit arbitrarily small perturbations $x'$ and $y'$ with $\supp\{x'\}\cap \supp\{y'\}\not=\emptyset$. This neighbourhood can also be described as
\[
N_\Gamma = \overline{\Gamma \times \Gamma - D(\Gamma,2)}.
\]
The boundary $\partial N_\Gamma\subset N_\Gamma$ consists of all pairs $(x,y)\in N_\Gamma$ admitting an arbitrarily small perturbation $(x',y')$ which does not lie in $N_\Gamma$.

The intersection form (\ref{intersection}) is defined as follows. Consider the injection $$j: \Gamma\times \Gamma \to (\Gamma\times \Gamma,D(\Gamma, 2))$$
and the induced homomorphism
$$j_\ast: H_2(\Gamma\times \Gamma) \to H_2(\Gamma\times \Gamma,D(\Gamma, 2))$$ on the two-dimensional homology. The group $H_2(\Gamma\times \Gamma)$
can be identified with $H_1(\Gamma)\otimes H_1(\Gamma)$ (by the K\"unneth theorem) and the group $H_2(\Gamma\times\Gamma, D(\Gamma, 2))$ can be identified with $H_2(N_\Gamma, \partial N_\Gamma)$ (by excision).
After these identifications $j_\ast$ turns into the homomorphism (\ref{intersection}).

The intersection form can also be described geometrically as follows. Let $z = \sum n_i e_i$ and $z' = \sum m_j e_j'$ be cycles in $\Gamma$, where $e_i$ and $e'_j$ are oriented edges of $\Gamma$. Then
\[
I(z \otimes z') = \sum_{(i,j) \in A} n_i m_j (e_i e'_j),
\]
where $A$ is the set of pairs $(i,j)$ such that $e_i \cap e'_j \neq \emptyset$. See \cite{BF}, \S 3.

For $\Gamma \not\cong S^1$, the intersection form enters the exact sequence
\begin{eqnarray}\label{secex}
\begin{array}{c}
 0\to H_2(F(\Gamma, 2)) \stackrel{\alpha_\ast}\to H_1(\Gamma) \otimes H_1(\Gamma) \stackrel {I_\Gamma}\to H_2(N_\Gamma, \partial N_\Gamma)\\ \\
 \stackrel \partial \to H_1(F(\Gamma, 2))\stackrel {\alpha_\ast} \to H_1(\Gamma\times\Gamma)\to 0.\end{array}\end{eqnarray}

It is convenient to introduce a shorthand notation
\begin{eqnarray}\label{qgamma}
Q_\Gamma= \coker I_\Gamma.
\end{eqnarray}
The sequence (\ref{secex}) gives a short exact sequence
\begin{eqnarray}\label{secq}
0\to Q_\Gamma \to H_1(D(\Gamma, 2)) \stackrel{\alpha_\ast}\to H_1(\Gamma\times \Gamma) \to 0,
\end{eqnarray}
and thus $Q_\Gamma$ can be regarded as a subgroup of $H_1(F(\Gamma, 2))$.
Note that the involution $\tau: F(\Gamma, 2) \to F(\Gamma, 2)$ given by $\tau(x,y) =(y,x)$ acts on $(N_\Gamma, \partial N_\Gamma)$ and
for $z, z'\in H_1(\Gamma)$ one has $I_\Gamma (z\otimes z')=-\tau_\ast I_\Gamma(z'\otimes z)$, see \cite{BF}, Lemma 2.2. It follows that (\ref{secq}) is an exact sequence of
$\Z[\Z_2]$-modules with $\tau$ acting on $H_1(\Gamma)\otimes H_1(\Gamma)$ by $z\otimes z'\mapsto -z'\otimes z$.

The relevance of the intersection form $I_\Gamma$ to the problem of calculating the homology of the configuration space $F(\Gamma, 2)$ can be illustrated by the following statement (see \cite{BF}, Proposition 2.3):

\begin{proposition}\label{prop1} Let $\Gamma$ be a finite connected graph which is not homeomorphic to the circle. Then
 the group $H_2(F(\Gamma, 2))$ is isomorphic to the kernel of the intersection form
\begin{eqnarray}
H_2(F(\Gamma, 2)) \cong \ker (I_\Gamma)
\end{eqnarray}
and the group $H_1(F(\Gamma, 2))$ is isomorphic to the direct sum
\begin{eqnarray}\label{six}H_1(F(\Gamma, 2)) \cong Q_\Gamma \oplus H_1(\Gamma) \oplus H_1(\Gamma).
\end{eqnarray}
 \end{proposition}

 \begin{corollary}\label{cor2} For a graph $\Gamma$ as in Proposition \ref{prop1} one has
 \begin{eqnarray}
 b_1(F(\Gamma, 2)) = 2b_1(\Gamma) + \rk\,  Q_\Gamma,
 \end{eqnarray}
 and \begin{eqnarray}\label{btwo}
b_2(F(\Gamma, 2)) =  b_1(\Gamma)^2 -b_1(\Gamma) +1 + \rk\,  Q_\Gamma - \Sigma
\end{eqnarray}  where $$\Sigma =\sum_{v\in V(\Gamma)}(\mu(v)-1)(\mu(v)-2).$$
 \end{corollary}

 In the last formula $V(\Gamma)$ denotes the set of vertices of $\Gamma$ and $\mu(v)$ denotes the number of edges incident to $v\in V(\Gamma)$ (i.e. the valence of $v$). In the exceptional case, $\Gamma \cong S^1$, it is easy to see that $F(\Gamma,2) \simeq S^1$ so the first and second Betti numbers are $1$ and $0$ respectively.

 Corollary \ref{cor2} follows from Proposition \ref{prop1} and from Corollary 2.5 in \cite{BF} giving explicitly the rank of the group $H_2(N_\Gamma, \partial N_\Gamma)$.
 By Corollary \ref{cor2}, knowing the rank of the cokernel of the intersection form $I_\Gamma$ is equivalent to knowing the Betti numbers of $F(\Gamma, 2)$.

 It should be noted that in many examples the intersection form $I_\Gamma$ is epimorphic or has a small cokernel, see \cite{BF}. In this respect we may mention Theorem 7.3 from \cite{BF} which deals with the case of planar graphs with all vertices of valence $\ge 3$.

 \begin{corollary}\label{matureint}
 A graph is mature (see Definition \ref{defmature}) if and only if $Q_\Gamma=0$, i.e. if the intersection form $I_\Gamma$ is surjective.
 \end{corollary}

  This follows from the exact sequence (\ref{secex}).

  \section{Enlarging graphs, I}

Let $\Gamma'\subset \Gamma$ be a subgraph. This means that the set of vertices of $\Gamma'$ is contained in the set of vertices of $\Gamma$ and the set of edges of
$\Gamma'$ is a subset of the set of edges of $\Gamma$. Denote by $N_\Gamma$ and $N_{\Gamma'}$ the corresponding subcomplexes of $\Gamma\times \Gamma$ and $\Gamma'\times \Gamma'$ correspondingly. Recall that $N_\Gamma$ can be described as the union of all squares $ee'=e\times e'$ where $e, e'$ are (closed) edges of $\Gamma$ with $e\cap e'\not=\emptyset$. We see that $N_{\Gamma'}$ is naturally contained in $N_{\Gamma}$.

Similarly $\partial N_\Gamma$ is the union of all products $ve=v\times e$ and $ev=e\times v$ where $e\in E(\Gamma)$ and $v\in V(\Gamma)$ such that
$v$ is connected by an edge to one of the ends of $e$. Clearly, $\partial N_{\Gamma'}\subset \partial N_{\Gamma}$.

\begin{proposition}\label{prop11}
For a subgraph $\Gamma'\subset \Gamma$, the inclusion $$(N_{\Gamma'}, \partial N_{\Gamma'})\to (N_{\Gamma}, \partial N_{\Gamma})$$ induces a monomorphism
$$H_2(N_{\Gamma'}, \partial N_{\Gamma'})\to H_2(N_{\Gamma}, \partial N_{\Gamma}).$$
\end{proposition}
\begin{proof}
Consider the cellular chain complex $C_\ast(N_\Gamma, \partial N_\Gamma)$. The group $C_2(N_\Gamma, \partial N_\Gamma)$ is free abelian generated by ordered pairs $ee'$ where $e, e'$ are edges of $\Gamma$ such that $e\cap e'\not=\emptyset$. Moreover, since $N_\Gamma$ has dimension 2 the homology group $H_2(N_\Gamma, \partial N_\Gamma)$ coincides with the kernel of the boundary homomorphism
$\partial: C_2(N_\Gamma, \partial N_{\Gamma}) \to C_1(N_\Gamma, \partial N_{\Gamma})$.

The chain complex $C_\ast(N_{\Gamma'}, \partial N_{\Gamma'})$ admits a similar description and therefore the inclusion $(N_{\Gamma'}, \partial N_{\Gamma'})\to (N_{\Gamma}, \partial N_{\Gamma})$ induces a monomorphism of chain complexes
$C_\ast(N_{\Gamma'}, \partial N_{\Gamma'})\to C_\ast (N_{\Gamma}, \partial N_{\Gamma}).$ Therefore, the induced homomorphism
$H_2(N_{\Gamma'}, \partial N_{\Gamma'})\to H_2(N_{\Gamma}, \partial N_{\Gamma})$ is a mono\-morphism, as
the restriction of the chain homomorphism
$C_2(N_{\Gamma'}, \partial N_{\Gamma'})$ $\to C_2(N_{\Gamma}, \partial N_{\Gamma})$.
\end{proof}

\begin{corollary}\label{cor22}
For any two cycles $z, z'\in H_1(\Gamma')$ lying in a subgraph $\Gamma'\subset \Gamma$ the intersection $I(z\otimes z') \in H_2(N_{\Gamma'}, \partial N_{\Gamma'})$ vanishes if and only if the intersection of their images $I(i_\ast(z)\otimes i_\ast(z')) \in H_2(N_\Gamma, \partial N_\Gamma)$ vanishes.
Here $i$ denotes the inclusion $\Gamma' \to \Gamma$.
\end{corollary}

We use Corollary \ref{cor22} in the proof of the following theorem:

\begin{theorem}\label{enlarge1} (a) Let $\Gamma$ be obtained from a connected graph $\Gamma'$ by adding an edge $e$ such that $\Gamma' \cap e$ is a single point
(see Figure \ref{enlarging}, (a)). Then $H_2(F(\Gamma', 2)) \cong H_2(F(\Gamma, 2))$ and the difference $b_1(F(\Gamma,2))-b_1(F(\Gamma', 2))$ equals $2\mu(v)-2$ where $\mu(v)$ denotes the valence in $\Gamma'$ of the vertex $v$ which is incident to the newly attached edge $e$.
\begin{figure}%[h]
\begin{center}
\resizebox{12cm}{5.3cm}{\includegraphics[21,312][566,561]{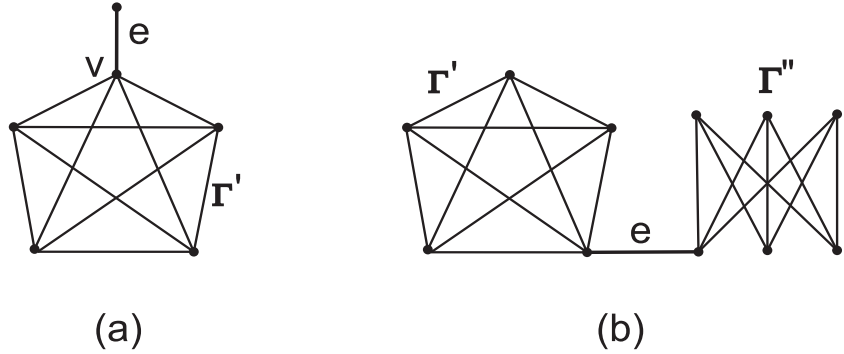}}
\end{center}
\caption{Enlarging graphs, I.}\label{enlarging}
\end{figure}

(b) Let $\Gamma$ be obtained from a graph $\Gamma'$ having two connected components
$\Gamma' = \Gamma'_1\sqcup \Gamma'_2$
by adding an edge $e$ connecting the components $\Gamma'_1$ and $\Gamma'_2$ (see Figure \ref{enlarging}, (b)). Then
\begin{eqnarray}\label{btwoagain}
\quad b_2(F(\Gamma, 2)) = b_2(F(\Gamma'_1,2)) + b_2(F(\Gamma'_2,2)) + 2b_1(\Gamma'_1)b_1(\Gamma'_2).
\end{eqnarray}
\end{theorem}
\begin{proof} The first part of statement (a) follows from Proposition \ref{prop1}, Corollary \ref{cor22} and from the following commutative diagram
$$
\begin{array}{ccc}
H_1(\Gamma')\otimes H_1(\Gamma') & \stackrel {I_{\Gamma'}}\to & H_2(N_{\Gamma'}, \partial N_{\Gamma'})\\
\downarrow \cong && \downarrow\\
H_1(\Gamma)\otimes H_1(\Gamma) & \stackrel {I_{\Gamma}}\to & H_2(N_{\Gamma}, \partial N_{\Gamma})
\end{array}
$$
where the vertical map on the left is an isomorphism induced by the inclusion $\Gamma'\to \Gamma$ and the vertical map on the right is injective according to Proposition \ref{prop11}. The second part of statement (a) is a consequence of Corollary 1.2 from \cite{BF}.

To prove statement (b) one notes that $H_1(\Gamma')\cong H_1(\Gamma'_1) \oplus H_1(\Gamma'_2)$ and therefore the tensor product $H_1(\Gamma')\otimes H_1(\Gamma')$ is the direct sum of the four groups $H_1(\Gamma'_1)\otimes H_1(\Gamma'_1)$, $H_1(\Gamma'_2)\otimes H_1(\Gamma'_2)$, $H_1(\Gamma'_1)\otimes H_1(\Gamma'_2)$ and $H_1(\Gamma'_2)\otimes H_1(\Gamma'_1)$. The intersection form $I_{\Gamma'}$ restricts as $I_{\Gamma'_1}$ and $I_{\Gamma'_2}$ on the first and the second summands correspondingly.
On the other hand the intersection form $I_{\Gamma'}$ vanishes on the two remaining summands. Now, taking into account Proposition \ref{prop1} we obtain formula (\ref{btwo}).
\end{proof}

If we are only interested in computing the second Betti number of $F(\Gamma, 2)$ then, by statement (a) of Theorem \ref{enlarge1}
 applied inductively, we may always simplify our graph by removing all \lq\lq freely attached trees\rq\rq.

\section{The linking homomorphism}

Let $\Gamma$ be a connected finite graph. Let $u, v\in V(\Gamma)$ be two vertices of $\Gamma$ which are not connected by an edge.
\begin{figure}[h]
\begin{center}
\resizebox{10cm}{5cm}{\includegraphics[41,423][467,663]{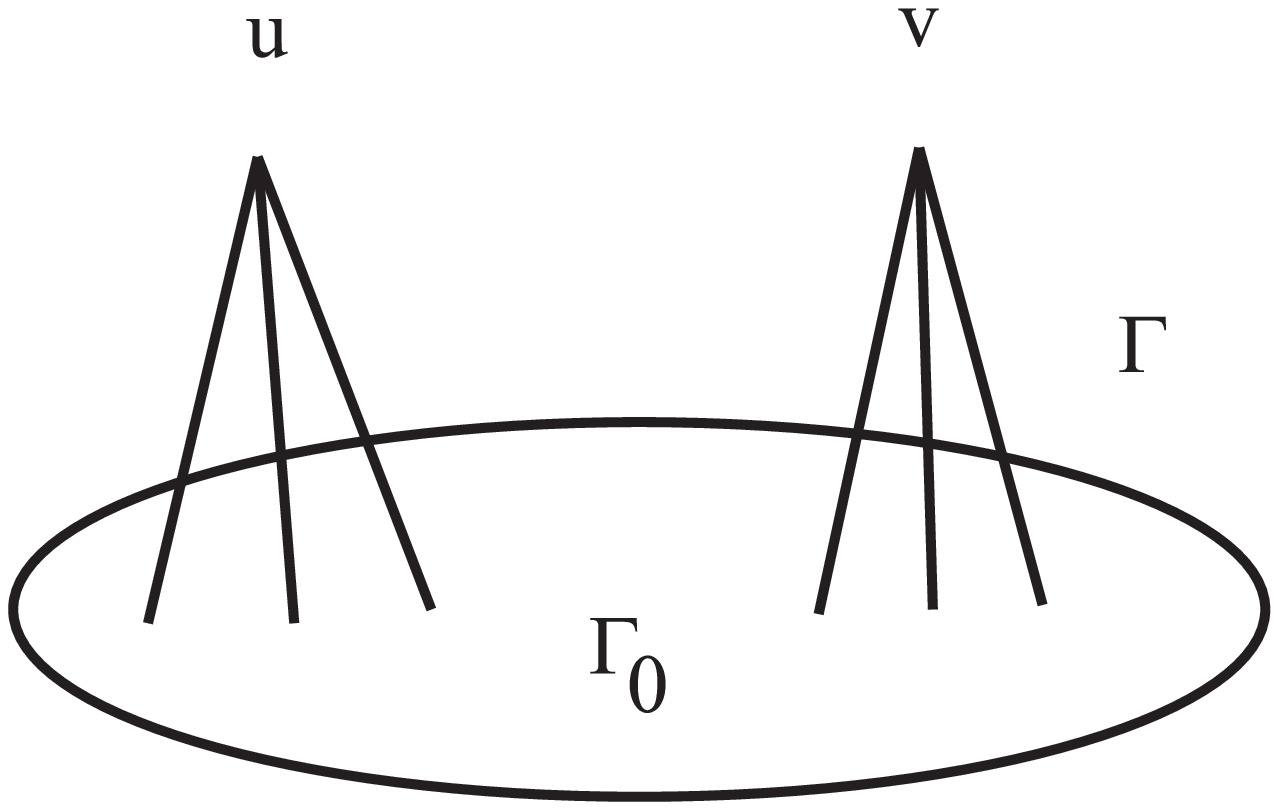}}
\end{center}
\caption{Linking.}\label{gamma}
\end{figure} We denote by $\Gamma_0$ the graph obtained from $\Gamma$ by removing $u$ and $v$ and all edges incident to $u, v$. Note that $\Gamma_0$ is connected if and only if $\Gamma-\{u, v\}$ is connected. Our goal is to examine the changes in the homology of $F(\Gamma, 2)$ when one attaches an edge connecting $u$ and $v$.

We define {\it a linking homomorphism}
\begin{eqnarray}\label{link}
\Lk_{v,u}: H_1(\Gamma_0) \to Q_\Gamma
\end{eqnarray}
where $Q_\Gamma$ is the cokernel of the intersection form (\ref{intersection}).
Consider a homology class $z\in H_1(\Gamma_0)$ and the corresponding cycle $c\in C_1(\Gamma_0)$. Since $\Gamma$ is connected we may find a chain
\begin{eqnarray}\label{aa}
a\in C_1(\Gamma) \, \, \mbox{with $\partial a=u-v$}.
\end{eqnarray}
Then $ac\in C_2(\Gamma\times \Gamma)$ is a chain satisfying
$$\partial (ac)=uc-vc.$$
Note that the boundary cycles $uc$ and $vc$ lie in the subcomplex $$C_\ast(D(\Gamma, 2))\subset C_\ast(\Gamma\times \Gamma)$$ and therefore $ac$ determines a cycle of the relative chain complex $$C_\ast(\Gamma\times \Gamma, D(\Gamma, 2))=C_\ast(N_\Gamma,\partial N_\Gamma).$$ We consider the homology class of $ac$ as an element of $H_2(N_\Gamma, \partial N_\Gamma)$ and denote by $\Lk_{v, u}(z)$ its image in $Q_\Gamma$.

\begin{proposition} The linking homomorphism $\Lk_{v,u}$ is well-defined, i.e. for $z\in H_1(\Gamma_0)$ the result $\Lk_{v,u}(z)\in Q_\Gamma$ does not depend on the choice of the chain $a$, see (\ref{aa}).
\end{proposition}
\begin{proof} If $a'\in C_1(\Gamma)$ is another chain satisfying $\partial a'=u-v$ then $$ac-a'c=(a-a')c\in C_2(\Gamma\times \Gamma)$$ is a cycle and its image under
$$H_2(\Gamma\times \Gamma) \to H_2(\Gamma\times \Gamma, D(\Gamma, 2))\cong H_2(N_\Gamma, \partial N_\Gamma)$$
lies in the image of the intersection form (\ref{intersection}) as it coincides with the intersection of cycles $a-a'$ and $c$. This shows that the difference  between the images of $ac$ and $a'c$
under $$C_2(\Gamma\times \Gamma) \to C_2(\Gamma\times\Gamma, D(\Gamma, 2))\cong C_2(N_\Gamma, \partial N_\Gamma)$$
lies in the image of $I_\Gamma$. Hence the coset of $ac$ in $Q_\Gamma$ is well-defined.
\end{proof}

To compute the linking form explicitly one may represent elements of $H_2(N_\Gamma, \partial N_\Gamma)$ as integer linear combinations of symbols
$ee'$ corresponding to ordered pairs of edges of $\Gamma$ with $e\cap e'\not=\emptyset$, see \cite{BF}, \S 3.

Consider the graph shown on Figure \ref{example1}.
\begin{figure}[h]
\begin{center}
\resizebox{7cm}{5cm}{\includegraphics[97,407][488,694]{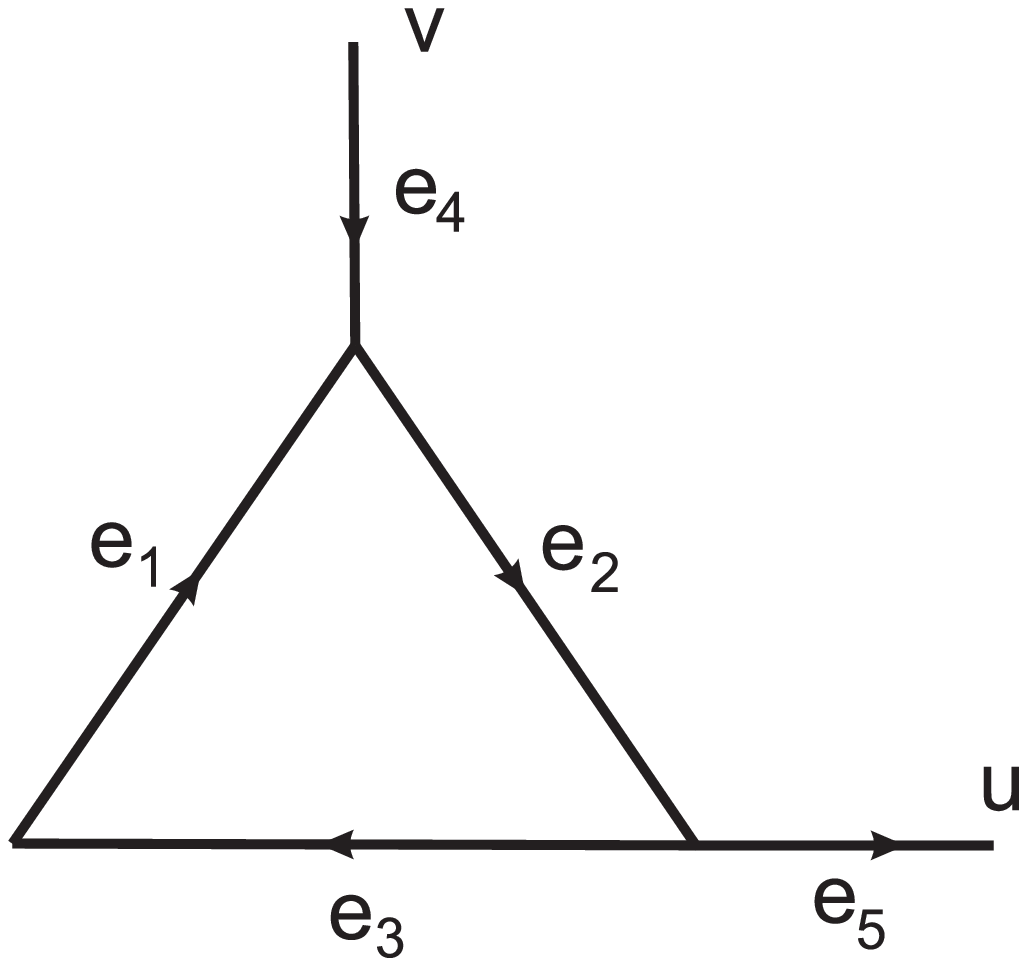}}
\end{center}
\caption{Linking.}\label{example1}
\end{figure}
The group $H_2(N_\Gamma, \partial N_\Gamma)$ is a subgroup of $C_2(N_\Gamma, \partial N_\Gamma)$ and the latter group is free abelian with basis $e_ie_j$ where
$i, j=1, \dots, 5$ with omission of the symbols $e_1e_5$, $e_5e_1$, $e_4e_3$, $e_3e_4$, $e_4e_5$, $e_5e_4$ corresponding to disjoint pairs of edges. If $z\in H_1(\Gamma_0)$ is the homology class of the cycle $c= e_1+e_2+e_3$ then $\Lk_{v, u}(z)$ is represented by the element $(e_4+e_2+e_5)(e_1+e_2+e_3)= e_4e_1+e_4e_2+e_2e_1+e_2e_2+e_2e_3+e_5e_2+e_5e_3\in H_2(N_\Gamma, \partial N_\Gamma)$ and one needs to take its coset in $Q_\Gamma$. In this example we have only one generator $z\in H_1(\Gamma)$ and
the intersection $I(z\otimes z)$ equals $(e_1+e_2+e_3)(e_1+e_2+e_3)$. We see that in this case the linking $\Lk_{v, u}(z)\in Q_\Gamma$ is nontrivial.

Next we mention several simple properties of the linking homomorphism.
Reversing the order of the vertices $(v,u)\mapsto (u, v)$ results in changing the sign of the linking homomorphism,
\begin{eqnarray}
\Lk_{v,u}(z) = - \Lk_{u,v}(z),
\end{eqnarray}
$z\in H_1(\Gamma_0)$. Besides,
\begin{eqnarray}
\Lk_{v,u}(z+z') = \Lk_{v, u}(z) + \Lk_{v, u}(z')
\end{eqnarray}
where $z, z'\in H_1(\Gamma_0)$.

\begin{lemma}\label{lmlink} Suppose that a homology class $z\in H_1(\Gamma_0)$ can be realized by a cycle $c=\sum n_ie_i \in C_1(\Gamma_0)$, $n_i\in \Z$, $n_i\not=0$ and there is
a chain $a=\sum m_je'_j \in C_1(\Gamma)$, $m_j\in \Z$, $m_j\not=0$ satisfying $\partial a=u-v$
and $e_i\cap e'_j=\emptyset$ for all $i, j$ (in other words, the chain $a$ connecting $v$ and $u$ and the cycle $c$ are disjoint). Then $\Lk_{v,u}(z)=0.$
\end{lemma}

This follows directly from the definition of $\Lk_{v,u}(z)$.

Lemma \ref{lmlink} shows that $\Lk_{v,u}(z)$ measures {\it\lq\lq linking phenomenon\rq\rq}\, (similar to the classical linking of a pair of disjoint closed curves in $\R^3$) between the cycle $z$ and the zero-dimensional sphere $S^0\subset \Gamma$ represented by the pair of vertices $v, u$.

\begin{corollary}
Suppose that the graph $\Gamma_0$ is disconnected and there are at least two connected components of $\Gamma_0$ which are connected by edges with both vertices $u$ and $v$.
Then $\Lk_{v, u}(z)=0$ for any $z\in H_1(\Gamma_0)$.
\end{corollary}
The Corollary is illustrated by Figure \ref{twoways}.
\begin{figure}[h]
\begin{center}
\resizebox{10cm}{5cm}{\includegraphics[21,407][534, 673]{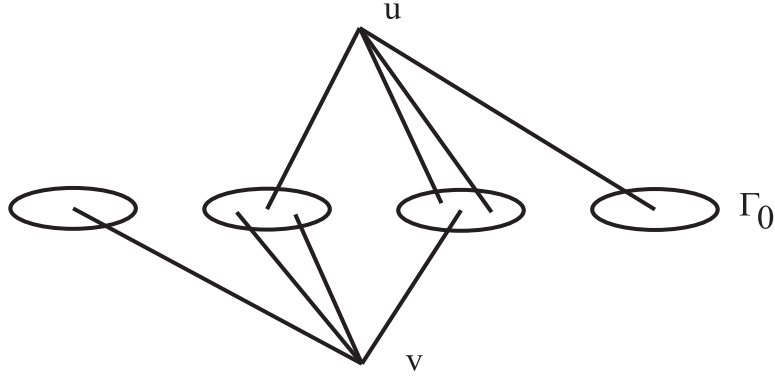}}
\end{center}
\caption{The case when $\Gamma_0$ is disconnected.}\label{twoways}
\end{figure}
\begin{proof} Let $\Gamma_0 = \sqcup_{i=1}^k \Gamma^i_{0}$ be the connected components of $\Gamma_0$. Then $H_1(\Gamma_0)= \oplus_{i=1}^k H_1(\Gamma^i_{0})$. It is enough to show that $\Lk_{v, u}(z)=0$ for any $z\in H_1(\Gamma^i_{0})$. Using our assumptions we see that given $i=1, \dots, k$ we may connect $u$ and $v$ by a path in $\Gamma$ which avoids
$\Gamma_0^i$. The result now follows from  Lemma \ref{lmlink}.
\end{proof}

Next we consider the embeddings
\begin{eqnarray}l_v, l_u: \Gamma_0\to D(\Gamma, 2),\end{eqnarray}
where for $x\in \Gamma_0$ one has
$$l_v(x)=(v, x), \quad l_u(x) =(u, x).$$
Here the symbol $l$ stands for \lq\lq left\rq\rq. The similar \lq\lq right\rq\rq\,  embeddings $r_u, r_v: \Gamma_0\to D(\Gamma, 2)$ are given by $$r_u(x)=(x, u), \quad r_v(x)=(x,v).$$

The following result gives a relation between the induced homomorphisms $(l_u)_\ast$, $(l_v)_\ast$ on homology and the linking homomorphism.

\begin{lemma}\label{parlink} For a homology class $z\in H_1(\Gamma_0)$ one has
\begin{eqnarray}\label{compos}
(l_u)_\ast(z) - (l_v)_\ast(z) = \partial(\Lk_{v, u}(z)) \, \, \in \, \, H_1(D(\Gamma, 2)),
\end{eqnarray}
where $\partial$ denotes the composition
\begin{eqnarray}\label{partial}
H_2(N_\Gamma, \partial N_\Gamma) \to H_1(\partial N_\Gamma) \to H_1(D(\Gamma, 2))
\end{eqnarray}
of the boundary homomorphism and the homomorphism induced by the inclusion $\partial N_\Gamma \to D(\Gamma, 2)$.
Similarly, one has
\begin{eqnarray}\label{compostau}
(r_u)_\ast(z) - (r_v)_\ast(z) = \tau_\ast(\partial(\Lk_{v, u}(z))) \, \, \in \, \, H_1(D(\Gamma, 2)),
\end{eqnarray}
where $$\tau: D(\Gamma, 2) \to D(\Gamma, 2)$$ is the involution acting by $\tau(x, y)=(y,x)$.
\end{lemma}
Note that the composition $\partial \circ \Lk_{v, u}$ (which appears in (\ref{compos})) is well-defined as follows from the exact sequence
$$\dots \to H_1(\Gamma)\otimes H_1(\Gamma) \stackrel{I_\Gamma}\to H_2(N_\Gamma, \partial N_\Gamma) \stackrel \partial \to H_1(D(\Gamma, 2))\to \dots,$$ see \cite{BF}, formula (12). This exact sequence also implies that
$\partial: Q_\Gamma  \to H_1(D(\Gamma, 2))$
is injective.

\begin{proof}[Proof of Lemma \ref{parlink}] The statement follows directly from our definitions. Indeed, given $z\in H_1(\Gamma_0)$ and a cycle
$c\in C_1(\Gamma_0)$ representing it, consider a chain $a\in C_1(\Gamma)$ with $\partial a= u-v$. Then the chain $ac\in C_2(\Gamma\times \Gamma, D(\Gamma, 2))$ is a cycle representing
$\Lk_{v, u}(z)$. The class $\partial \Lk_{v, u}(z)\in H_1(D(\Gamma, 2))$ is represented by the cycle $\partial(ac) = uc-vc$ which equals $(l_u)_\ast(z) - (l_v)_\ast(z)$.

The second statement follows from the first since $r_u=\tau\circ l_u$ and $r_v=\tau\circ l_v$.
\end{proof}

% \begin{example} Consider the graph $\Gamma$ shown on Figure \ref{figlink} on the left. Note that $\Gamma$ is planar and satisfies the conditions of Theorem 7.3 from \cite{BF}.
% Hence $\coker I_\Gamma$ is nontrivial and has rank one, see the footnote on page 123 of \cite{BF}.
% The graph $\Gamma_0$ shown on the right is obtained from $\Gamma$ by removing vertices $u, v$ and all edges incident to them.
% \begin{figure}[h]
% \begin{center}
% \resizebox{12cm}{5cm}{\includegraphics[14,526][591,792]{figlink.eps}}
% \end{center}
% \caption{Example with nontrivial linking.}\label{figlink}
% \end{figure}
% It is clear that the linking homomorphism vanishes on the cycle $ABCD$ and is nontrivial on the cycle $CDE$. The dotted line shows a path connecting $u$ and $v$ in $\Gamma$.
% \end{example}

\section{Enlarging graphs, II}

In this section we use the linking homomorphism introduced in the previous section to describe the homology of the configuration space of two particles on a graph in the situation when a new edge is added to the graph.

Let $\Gamma$ be a finite connected graph not homeomorphic to $[0,1]$ and let $u, v \in V(\Gamma)$ be two vertices which are not connected by an edge in $\Gamma$. The new graph $\hat \Gamma$ is obtained from $\Gamma$ by adding an edge $e$ connecting $u$ and $v$, see Figure \ref{gamma1}. The inclusion $\Gamma\to \hat \Gamma$ gives an embedding of the configuration spaces
$F(\Gamma, 2) \to F(\hat \Gamma, 2)$. Denote by $\Gamma_0$ the graph obtained from $\Gamma$ by removing $u$, $v$ and all edges emanating from $u$ and $v$.

\begin{figure}[h]
\begin{center}
\resizebox{8cm}{4cm}{\includegraphics[33,419][461, 671]{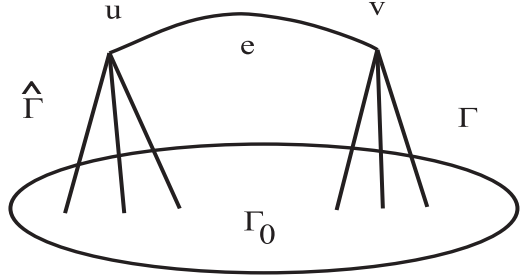}}
\end{center}
\caption{Enlarging graphs, II.}\label{gamma1}
\end{figure}

\begin{theorem}\label{thm10} In the notations introduced above, consider the inclusion $i: F(\Gamma, 2) \to F(\hat \Gamma, 2)$ and the involution $\tau(x, y)=(y,x)$ acting on the spaces
$F(\Gamma, 2)$ and $F(\hat \Gamma, 2)$. The homology groups $H_\ast(F(\Gamma, 2))$ and $H_\ast(F(\hat \Gamma, 2))$ are $\Z[\Z_2]$-modules via the induced action $\tau_\ast$.

(a) There exists
a short exact sequence
of $\Z[\Z_2]$-modules
\begin{eqnarray}
0\to H_2(F(\Gamma, 2)) \stackrel{i_\ast}\to H_2(F(\hat \Gamma, 2))\to X\to 0
\end{eqnarray}
where $X$ is the subgroup of $H_1(\Gamma_0)\oplus H_1(\Gamma_0)$ consisting of elements $z\oplus z'$ such that $$\Lk_{v, u}(z) +\tau_\ast(\Lk_{v, u}(z'))=0\in Q_\Gamma.$$ The involution $\tau$ acts on $X$ by $\tau(z\oplus z')= z'\oplus z$.

(b) The relations between the one-dimensional homology of $F(\Gamma, 2)$ and
$F(\hat\Gamma, 2)$ can be described by an exact sequence
$$H_1(\Gamma_0)\oplus H_1(\Gamma_0) \stackrel F\to H_1(F(\Gamma, 2)) \stackrel{i_\ast}\to H_1(F(\hat \Gamma, 2)) \stackrel k\to  H_0(\Gamma_0)\oplus  H_0(\Gamma_0)\to 0$$
where $F$ acts by the formula
\begin{eqnarray}\label{fff}
F(z\oplus z')= \partial_\ast \Lk_{v, u}(z)+\tau_\ast \partial_\ast \Lk_{v, u}(z'), \quad z, z'\in H_1(\Gamma_0).\end{eqnarray}
\end{theorem}
\begin{proof} We will deal with the discrete configuration spaces $D(\Gamma, 2)$ and $D(\hat \Gamma, 2)$ instead of $F(\Gamma, 2)$ and $F(\hat \Gamma, 2)$.
It will be convenient to denote $D=D(\Gamma, 2)$, $\hat D = D(\hat \Gamma, 2)$. One has
\begin{eqnarray}
\hat D = D\cup e\Gamma_0\cup \Gamma_0 e,
\end{eqnarray}
where $e\Gamma_0$ and $\Gamma_0 e$ denote the cartesian products $e\times \Gamma_0, \, \, \Gamma_0 \times e\, \subset D(\hat \Gamma, 2).$
Note that $e\Gamma_0\cap D = u\Gamma_0 \cup v\Gamma_0$ and similarly $\Gamma_0e \cap D = \Gamma_0u \cup \Gamma_0v$. Besides, $e\Gamma_0\cap \Gamma_0e=\emptyset$.
The involution $\tau$ maps $e\Gamma_0$ onto $\Gamma_0e$ and vice versa. Thus we obtain
\begin{eqnarray*}
H_i(\hat D, D) = H_i((e,\partial e)\times \Gamma_0) \oplus H_i(\Gamma_0\times (e, \partial e))=
H_{i-1}(\Gamma_0) \oplus H_{i-1}(\Gamma_0)
\end{eqnarray*}
and we obtain the following long exact sequence of $\Z[\Z_2]$-modules
\begin{eqnarray*}
0\to &H_2(D)& \stackrel{i_\ast}\to H_2(\hat D) \stackrel{k}\to H_1(\Gamma_0)\oplus H_1(\Gamma_0) \\
 \stackrel{F}\to  &H_1(D)& \stackrel{i_\ast}\to H_1(\hat D) \stackrel{k}\to  H_0(\Gamma_0)\oplus H_0(\Gamma_0)\to 0.
\end{eqnarray*}
The homomorphism $F$ acts as follows: for $z, z' \in H_1(\Gamma_0)$ one has
$$F(z\oplus z') = ({l_u}_\ast-{l_v}_\ast)(z)\oplus ({r_u}_\ast-{r_v}_\ast)(z'),$$
where $l_u, l_v, r_u, r_v$ are defined before Lemma \ref{parlink}. Applying Lemma \ref{parlink} we obtain formula (\ref{fff}).
\end{proof}

We can describe the homomorphism $k: H_1(\hat D) \to H_0(\Gamma_0) \oplus H_0(\Gamma_0)$ (which appears in the proof above) explicitly, as follows. Given a vertex $w\in \Gamma_0$ consider the loop $\alpha_w: S^1\to \hat D$ of the following form. Represent $S^1$ as the union of two arcs $S^1=A\cup B$
and define $\alpha_w|A$ as the path in $D(\hat \Gamma, 2)$ where the first point stays constantly at $w$ and the second point travels along the edge $e$ from $u$ to $v$. The restriction
$\alpha_w|B$ is  a path in $D(\Gamma, 2)$ which starts at $(w, v)$ and ends at $(w, u)$. The homomorphisms $k$ sends the homology class of the loop $\alpha_w$ to
$1_w\oplus 0 \in H_0(\Gamma_0)\oplus H_0(\Gamma_0)$ where $1_w\in H_0(\Gamma_0)$ denotes the class represented by the connected component of $w$.

Since homology classes of loops of type $\alpha_w$ and their images under the involution $\tau$ generate the cokernel of $i_\ast: H_1(D)\to H_1(\hat D)$, this description fully defines $k$.

\begin{corollary}
 Suppose that under the assumptions of Theorem \ref{thm10} one knows that
 the linking homomorphism $\Lk_{v, u}: H_1(\Gamma_0) \to Q_\Gamma$ vanishes.
 Then one has the following exact sequence of $\Z[\Z_2]$-modules
 \begin{eqnarray*}
%\hspace{2cm}
0\to H_r(F(\Gamma, 2)) \stackrel{i_\ast}\to H_r(F(\hat \Gamma, 2)) \to H_{r-1}(\Gamma_0) \oplus H_{r-1}(\Gamma_0)\to 0
 \end{eqnarray*}
 where $r=1, 2$ and the involution $\tau$ acts on $H_{r-1}(\Gamma_0) \oplus H_{r-1}(\Gamma_0) $ by interchanging the summands.
\end{corollary}

\section{Proof of Theorem \ref{thm0}}

In this section we consider again the process of adding an edge to a graph and examine its effect on the group $Q_\Gamma$.

\begin{theorem} \label{thm14} Let $\Gamma$ be a finite connected graph homeomorphic to neither $[0,1]$ nor $S^1$. Let $u,v$ be two vertices in $\Gamma$ that are not joined by an edge and let $\hat \Gamma$ be obtained from $\Gamma$ by adding an edge joining $u$ and $v$. Then one has the following exact sequence
$$0\to (A+\tau A)\to Q_\Gamma \to Q_{\hat\Gamma} \to G\to 0,$$
where $A\subset Q_\Gamma$ denotes the image of the linking homomorphism
$$\Lk_{v, u}: H_1(\Gamma_0) \to Q_\Gamma$$ and $G$ is a free abelian group of rank $2b_0(\Gamma_0)-2$. Here $\Gamma_0$ is obtained from $\Gamma$ by removing $u, v$ and all edges incident to these vertices.
\end{theorem}
\begin{proof} Consider the commutative diagram
$$
\begin{array}{ccccccccc}
&&0&& 0 &&&&\\
&&\downarrow && \downarrow &&&&\\
&&A+\tau A&&A+\tau A && 0&&\\
&&\downarrow &&\downarrow &&\downarrow&&\\
0 &\to & Q_\Gamma & \to & H_1(F(\Gamma, 2)) & \stackrel{\alpha_\ast}\to & H_1(\Gamma\times \Gamma) & \to & 0\\
&&\, \, \downarrow \gamma && \downarrow &&\downarrow&&\\
0 &\to & Q_{\hat \Gamma} & \to & H_1(F(\hat\Gamma, 2)) & \stackrel{\alpha_\ast}\to & H_1(\hat \Gamma\times \hat \Gamma) & \to & 0\\
&&\downarrow &&\downarrow&& \downarrow &&\\
0&\to&G&\to& H_0(\Gamma_0)\oplus H_0(\Gamma_0)& \stackrel \beta\to& \Z\oplus \Z& \to & 0\\
&&\downarrow &&\downarrow &&\downarrow&&\\
&&0&& 0 && 0&&
\end{array}
$$
The vertical exact sequence in the middle is given by statement (b) of Theorem \ref{thm10}. The vertical exact sequence on the right is obvious: adding an edge adds a summand $\Z$ to the first homology group of the graph. This gives the vertical exact sequence on the left; in other words, the kernel of $\gamma$ is $A+\tau A$ and the cokernel of $\gamma$ is a free abelian group of rank $2b_0(\Gamma_0)-2$, as claimed.
\end{proof}

Note that the condition that there is no edge joining $u$ and $v$ can be achieved by subdividing any edges between $u$ and $v$, if necessary.

\begin{corollary}\label{cor15}
Let $\hat \Gamma= \Gamma\cup e$ be obtained from a finite connected graph $\Gamma$ (which is homeomorphic to neither $[0,1]$ nor $S^1$) by adding an edge $e$ attached to two vertices
$\{u, v\}$ such that there is no edge between $u$ and $v$ in $\Gamma$ and $\Gamma-\{u, v\}$ is connected. If $Q_\Gamma=0$ then $Q_{\hat \Gamma}=0$.
\end{corollary}

Corollary \ref{cor15} is equivalent to Theorem \ref{thm0}.

It would be interesting to characterize all graphs which can be obtained from the Kuratowski graphs $K_5$ and $K_{3,3}$ by subdivision and by subsequent adding of edges
satisfying the condition of Corollary \ref{cor15}.

\begin{figure}[h]
\begin{center}
\resizebox{6cm}{5cm}{\includegraphics[69,274][489,668]{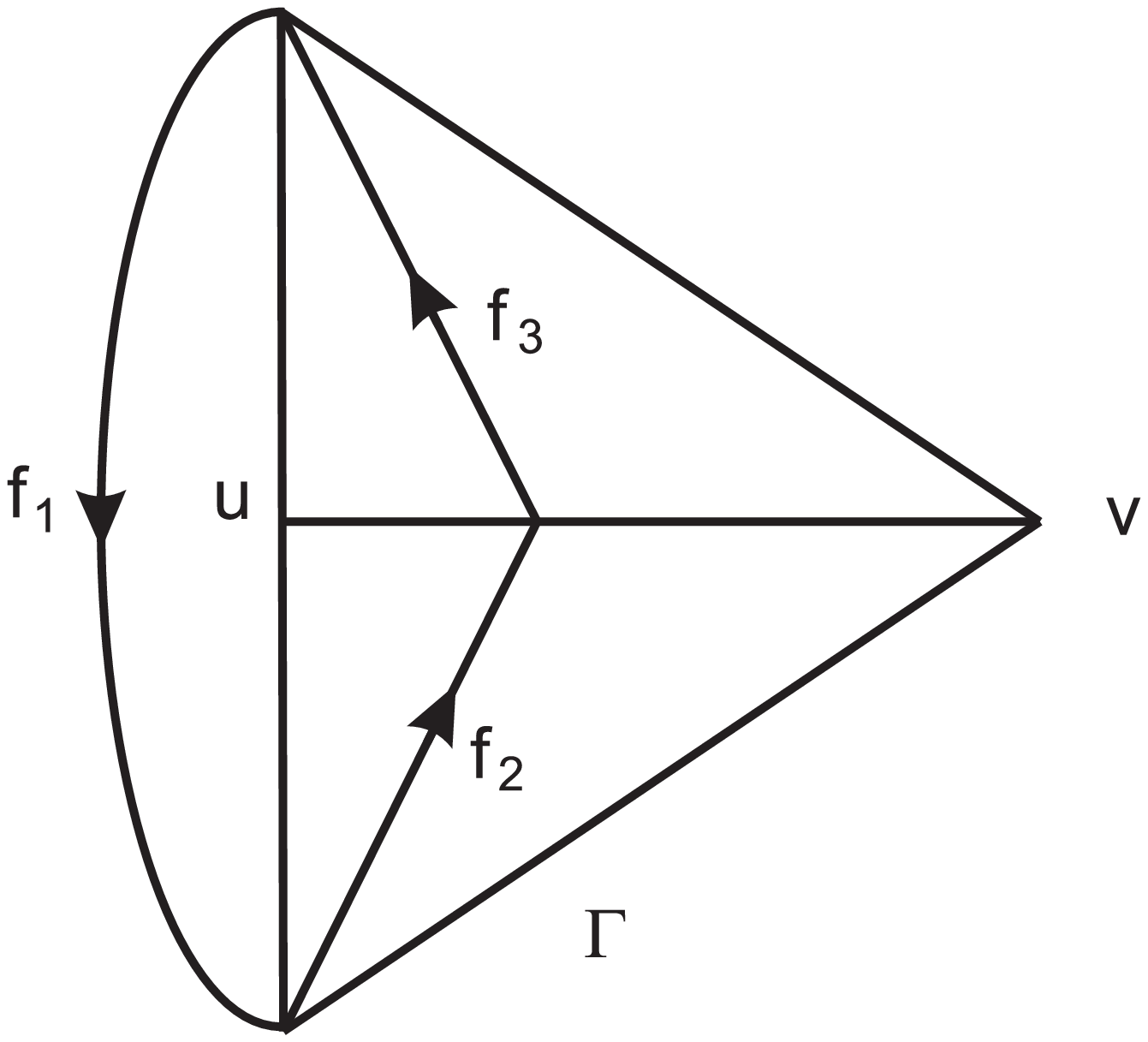}}
\end{center}
\caption{Graph $\Gamma$ with $Q_\Gamma=\Z$.}\label{gamma4}
\end{figure}

\begin{example} Consider the graph $\Gamma$ shown in Figure \ref{gamma4}. It is planar and satisfies the conditions of Theorem 7.3  from \cite{BF}; thus by Corollary 7.4 from \cite{BF} one has $Q_\Gamma=\Z$. Removing the vertices $u$ and $v$ does not disconnect the graph.
Let $\hat \Gamma$ be the result of adding an edge $e$ connecting $u$ and $v$. Clearly one has $\hat \Gamma = K_5$ and therefore $Q_{\hat\Gamma}=0$.
Applying Theorem \ref{thm14} we see that the subgroup $A+\tau A$ is isomorphic to $\Z$ in this case, i.e. $A=\tau A = \Z$.
One also sees that $u$ and $v$ are linked with respect to the triangular cycle in $\Gamma$ and the equation $A=\tau A = \Z$ can be confirmed by a direct calculation.
\end{example}

\begin{figure}[h]
\begin{center}
\resizebox{5cm}{4cm}{\includegraphics[95,324][489,668]{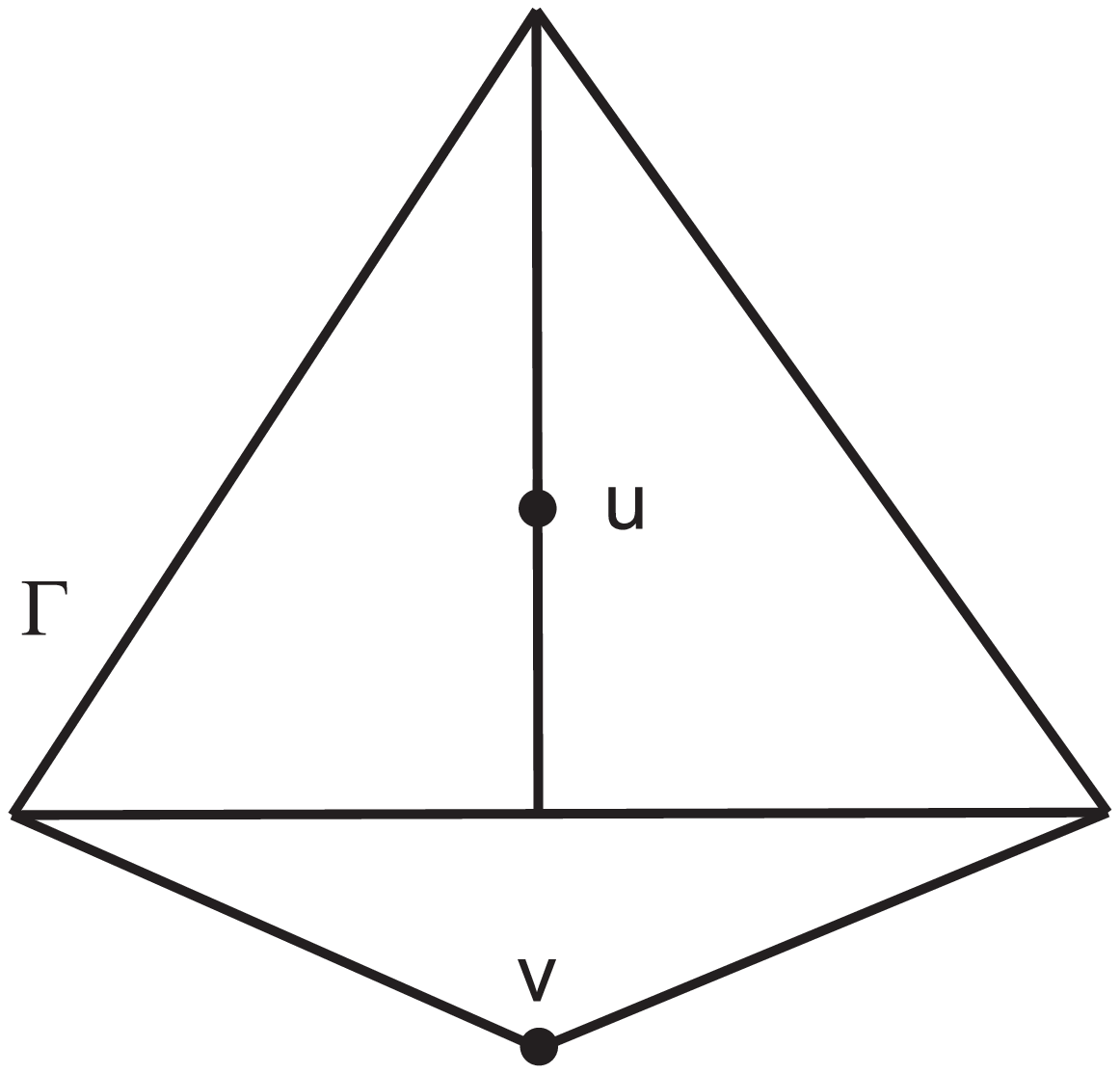}}
\end{center}
\caption{Graph $\Gamma$ with $Q_\Gamma=\Z$.}\label{gamma3}
\end{figure}

At this point, we may compare Theorem \ref{thm14} (in conjunction with formula (\ref{btwo})) with Theorem 4.2 of \cite{Patty62}. In \cite{Patty62}, the author considers attaching a new edge to a graph by gluing each of its endpoints to a univalent vertex. Theorem 4.2 of \cite{Patty62} implies in particular that the resulting change in the second Betti number of the configuration space is always an even number. Let us test this statement in the case of the graph $\hat{\Gamma}$ we considered above. Note that $\hat{\Gamma}$ may be obtained from $\Gamma$ by first attaching two \lq\lq free\rq\rq\, edges (as in Theorem \ref{enlarge1}, (a)), one with $\Gamma \cap e_1 = u$, one with $\Gamma \cap e_2 = v$ and then attaching an edge joining the two resulting univalent vertices, in the manner of \cite{Patty62}, Theorem 4.2. The first of these steps does not alter $b_2(F(\Gamma,2))$, by Theorem \ref{enlarge1}. The graph $\Gamma$ has 
$b_2(F(\Gamma,2))=0$ by Theorem 6.1 in \cite{BF} and the graph $\hat{\Gamma}$ has $b_2(F(\hat{\Gamma},2) )=1$, by Example 4.1 in \cite{BF}. 
Thus we conclude that Theorem 4.2 of \cite{Patty62} cannot be correct. 

As another example illustrating Theorem \ref{thm14}, consider the graph $\Gamma$ shown on Figure \ref{gamma3}: one has $Q_\Gamma=\Z$ and adding an edge connecting the vertices $u, v$ produces a graph
$\hat \Gamma = K_{3,3}$; we know that $Q_{\hat \Gamma}=0$, i.e. $K_{3,3}$ is mature. In this case we also have $A=\Z$ and $A=\tau A$.

\section{Examples of mature and non-mature graphs}

%In this section we consider examples of mature and non mature graphs.

\begin{proposition}
\label{univalent}
Let $\Gamma$ be a graph having a univalent vertex. Then $\Gamma$ is not mature.
\end{proposition}
\begin{proof}
We may assume that $\Gamma$ is not homeomorphic to $[0,1]$. Suppose that $v$ is a univalent vertex of a graph $\Gamma$
which is incident to an edge $e$.
Let $u$ be the other vertex incident to $e$. Without loss of generality we may assume that $\mu_\Gamma(u)\ge 3$; in the case $\mu_\Gamma(u)=2$ we may change the subdivision of the graph and amalgamate it with the following edges. Here $\mu_\Gamma(u)$ denotes the number of edges of $\Gamma$ incident to $u$. 
Applying Theorem \ref{enlarge1} (the second statement of part (a)) and equation (\ref{six}) of Proposition \ref{prop1}, we see that $Q_\Gamma$ has rank $\ge 2\mu_\Gamma(u)-4\ge 2$ which implies that $\Gamma$ is not mature. See Figure
\ref{nonexamples}, left.
\end{proof}
\begin{figure}[h]
\begin{center}
\resizebox{12cm}{5cm}{\includegraphics[17,483][561,714]{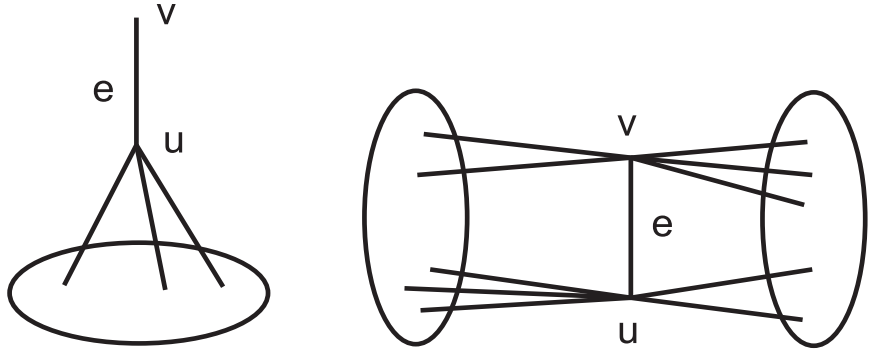}}
\end{center}
\caption{Non-mature graphs.}\label{nonexamples}
\end{figure}

\begin{proposition} 
\label{prop: non-mature}
Let $\Gamma$ be a graph such that removing the closure of an edge makes it disconnected. Then $\Gamma$ is not mature.
\end{proposition}
\begin{proof}
Let $e$ be an edge of $\Gamma$ with end points $u, v$. Let $\Gamma'$ denote the result of removing the interior of $e$ from $\Gamma$, see Figure
\ref{nonexamples}, right. Then $\Gamma$ is obtained from $\Gamma'$ by adding an edge with end points $u, v\in \Gamma'$. Suppose first that $\Gamma'$ is connected. Since $\Gamma'-\{u,v\}$ is disconnected, Theorem \ref{thm14} implies that $Q_\Gamma$ contains a free abelian subgroup of rank $2b_0(\Gamma'-\{u, v\})-2 \ge 2$. Thus $\Gamma$ is not mature. Now suppose that $\Gamma'$ is not connected. Then $\Gamma$ and $\Gamma'$ are related as in Theorem \ref{enlarge1}(b). Equations (\ref{btwo}) and (\ref{btwoagain}) imply that
\[
\mathrm{rk} \, Q_\Gamma = \mathrm{rk} \, Q_{\Gamma'_1} + \mathrm{rk} \, Q_{\Gamma'_2} + 2(\mu(u) - 1) + 2(\mu(v) -1) + 1.
\]
Here $\Gamma'_1$ and $\Gamma'_2$ are as in Theorem \ref{enlarge1}(b) and $\mu(u)$ and $\mu(v)$ denote the valences of $u$ and $v$ in $\Gamma'$. Thus we again see that $\Gamma$ is not mature.
\end{proof}

We say that a graph has a \emph{double edge} if it is homeomorphic to a 1-dimensional cell complex containing a pair of edges with the same endpoints. If $\Gamma$ is a graph having a double edge then $\Gamma$ is homeomorphic to a graph which can be disconnected by removing the closure of an edge. Thus we have the following:

\begin{proposition}
\label{double edge}
If $\Gamma$ is a graph with a double edge then $\Gamma$ is not mature.
\end{proposition}

A graph $\Gamma$ is {\it a wedge} of two subgraphs $\Gamma_1, \Gamma_2\subset \Gamma$ if  $\Gamma = \Gamma_1\cup \Gamma_2$ and the intersection
$\Gamma_1\cap \Gamma_2$ is a single vertex.

We say that a graph $\Gamma$ is {\it a double wedge} of two subgraphs $\Gamma_1, \Gamma_2\subset \Gamma$ if  $\Gamma = \Gamma_1\cup \Gamma_2$ and the intersection
$\Gamma_1\cap \Gamma_2$ consists of two vertices, see Figure \ref{gamma9}.

%\begin{figure}
%\begin{center}
%\includegraphics{gamma9.eps}
%\end{center}
%\caption{Wedge (a) and double wedge (b).}\label{gamma9}
%\end{figure}

%\begin{figure}
%\begin{center}
%\resizebox{13cm}{4cm}{\includegraphics[24,512][565,671]{wedge.eps}}
%\end{center}
%\caption{Wedge (left) and double wedge (right).}\label{gamma9}
%\end{figure}
\begin{figure}
\begin{center}
\resizebox{13cm}{4cm}{\includegraphics[20,512][565,671]{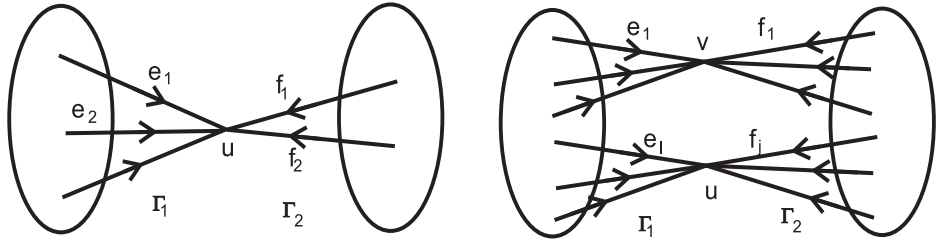}}
\end{center}
\caption{Wedge (left) and double wedge (right).}\label{gamma9}
\end{figure}

\begin{proposition}\label{prop18}
\label{maturity v connectedness}
If a graph $\Gamma$ is a wedge or a double wedge of two subgraphs $\Gamma_1, \Gamma_2\subset \Gamma$ such that each of $\Gamma_1$ and $\Gamma_2$ is connected and not homeomorphic to $[0,1]$, then $\Gamma$ is not mature.
\end{proposition}

% First suppose that $\Gamma = \Gamma_1 \vee \Gamma_2$. Then $\Gamma$ contains an edge $e$ such that $\Gamma - \bar{e}$ is disconnected so $\Gamma$ is not mature by Proposition \ref{prop: non-mature}.
%Recall that a vertex is essential if it is incident to at least three edges.

%In Proposition \ref{prop18} we do not allow $\Gamma_i$ to be homeomorphic to $[0,1]$.
%If one of the graphs $\Gamma_1$ or $\Gamma_2$ is homeomorphic to $[0,1]$ then $\Gamma$ contains a univalent vertex and thus $\Gamma$ is immature by Proposition \ref{univalent}.
\begin{proof} Note that if $\Gamma$ decomposes as a wedge, it can be disconnected by removing the closure of an edge and hence is not mature by Proposition \ref{prop: non-mature}. We offer an alternative proof of non-maturity below. 

We will use Corollary \ref{matureint} and show that for the graphs mentioned in the Proposition the intersection form $I_\Gamma$ is not surjective.

We assume that the intersection $\Gamma_1\cap \Gamma_2$ consists either of one (case one) or of two (case two) vertices.

Denote by $e_i$ and $f_j$ the edges of $\Gamma_1$ and $\Gamma_2$ respectively which are incident to the intersection $\Gamma_1\cap \Gamma_2$.
We will assume that $e_i$ and $f_j$ are oriented towards the vertices of $\Gamma_1\cap \Gamma_2$, see Figure \ref{gamma9}.

Recall that for every pair of oriented edges $e$ and $e'$ of $\Gamma$ with $e\cap e'\not=\emptyset$ one has defined the cohomology class
$$\{f_{ee'}\}\in H^2(N_\Gamma, \partial N_\Gamma)= \Hom(H_2(N_\Gamma, \partial N_\Gamma); \Z), $$
see \cite{BF}, \S 5.
Thus we may consider the homomorphism
\begin{eqnarray}
J= \sum \{f_{e_if_j}\}: H_2(N_\Gamma, \partial N_\Gamma) \to \Z,
\end{eqnarray}
where summation is taken over all pairs $e_i$ and $f_j$ with $e_i\cap f_j\not=\emptyset$. Intuitively, given two cycles $z, z'\in H_1(\Gamma)$, the number $J(I_\Gamma(z\otimes z'))$
\lq\lq counts instances\rq\rq\,  when $z$ and $z'$ are close to each other and
$z$ lies in $\Gamma_1$ and $z'$ lies in $\Gamma_2$.
We claim that:

{\it (a) in case one $J$ vanishes on the image of the intersection form $$I_\Gamma: H_1(\Gamma) \otimes H_1(\Gamma) \to H_2(N_\Gamma, \partial N_\Gamma);$$

 (b) in case two $J$ takes even values on the image of $I_\Gamma$.}

 Note that in case one, $H_1(\Gamma)= H_1(\Gamma_1)\oplus H_1(\Gamma_2)$. Similarly, in case two one has
 $H_1(\Gamma)= H_1(\Gamma_1)\oplus H_1(\Gamma_2)\oplus \Z$
 where the additional summand $\Z$ is represented by
  a cycle $z_0$ that is the union of a path in $\Gamma_1$ from $u$ to $v$ and a path in $\Gamma_2$ from $v$ to $u$; here $\Gamma_1\cap \Gamma_2= \{u, v\}$.

  In case one, examining $J(I_\Gamma(z\otimes z'))\in \Z$ for $z, z'$ lying in $H_1(\Gamma_1)$ or in $H_1(\Gamma_2)$ (four cases) one obtains (a).
  In case two, one has to consider the number $J(I_\Gamma(z\otimes z'))$ for $z, z'$ lying in $H_1(\Gamma_1)$ or in $H_1(\Gamma_2)$ or for $z, z'$ being equal to $z_0$ (nine cases in total); the only nonzero result is $J(I_\Gamma(z_0\otimes z_0))=-2$. Thus (b) follows.

  In view of (a) and (b) Proposition \ref{prop18} follows once we show that there exists a homology class $a\in H_2(N_\Gamma, \partial N_\Gamma)$ with $J(a)=1$.
  \begin{figure}[h]
\begin{center}
\resizebox{9cm}{4.2cm}{\includegraphics[8,340][574,576]{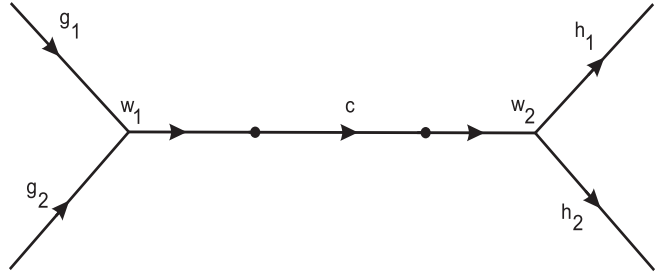}}
\end{center}
\caption{Construction of $a\in H_2(N_\Gamma, \partial N_\Gamma)$.}\label{tail}
\end{figure}
  Choose $w_i\in \Gamma_i$ to satisfy one of the following (i) $w_i$ has valence $\geq 3$ in $\Gamma_i$ or (ii) $w_i \in \{u,v\}$ and $w_i$ has valence $\geq 2$ in $\Gamma_i$. We do not exclude the case that $w_1 = w_2$. The assumption that $\Gamma_1, \Gamma_2 \ncong [0,1]$ implies we can always choose such $w_i$. Let $c$ be a simple path
  in $\Gamma$ connecting $w_1$ to $w_2$. We view $c$ as a cellular chain $c\in C_1(\Gamma)$ with $\partial c= w_2-w_1$. Consider the edges $g_1, g_2 \in \Gamma_1$, $h_1, h_2 \in \Gamma_2$ incident
  to $w_1, w_2$, as shown on Figure \ref{tail}. We may assume, by subdividing if necessary, that the graph shown in Figure \ref{tail} is embedded in $\Gamma$. The product
  $$a=(g_1+c+h_1)(g_2+c+h_2) \in C_2(\Gamma\times \Gamma, D(\Gamma, 2))=C_2(N_\Gamma, \partial N_\Gamma)$$
  is a relative cycle and obviously $J(a)=1$. Here $D(\Gamma, 2)$ is the discrete configuration space, see \S 2.
  This completes the proof.\end{proof}

Note that statement (b) from the proof becomes false for triple and higher order wedges.

% The following statement is a special case of Proposition \ref{prop18}:
% \begin{proposition}
% \label{prop: non-mature}
% Let $\Gamma$ be a graph such that removing the closure of an edge makes it disconnected. Then $\Gamma$ is not mature.
% \end{proposition}
%\begin{proof}
%Let $e$ be an edge of $\Gamma$ with end points $u, v$. Let $\Gamma'$ denote the result of removing the interior of $e$ from $\Gamma$. Then $\Gamma$ is obtained from $\Gamma'$ by adding an edge with end points $u, v\in \Gamma'$. Since $\Gamma'-\{u, v\}$ is disconnected, Theorem \ref{thm14} implies that $C_\Gamma$ contains a free abelian subgroup of rank $2b_0(\Gamma'-\{u, v\})-2 \ge 2$. See Figure
%\ref{nonexamples}, right.
%\end{proof}

As a useful result producing mature graphs we may mention the following.

\begin{proposition}\label{addition} Assume that $\Gamma=\Gamma'\cup \Gamma''$ is the union of two mature subgraphs such that
the edges incident to any vertex $v \in \Gamma'\cap \Gamma''$
lie either  all in $\Gamma'$ or all in $\Gamma''$ (see Figure \ref{example2}). If the intersection $\Gamma'\cap \Gamma''$ is connected then $\Gamma$ is mature.
\end{proposition}
\begin{figure}
\begin{center}
\resizebox{12cm}{4cm}{\includegraphics[8,367][582,573]{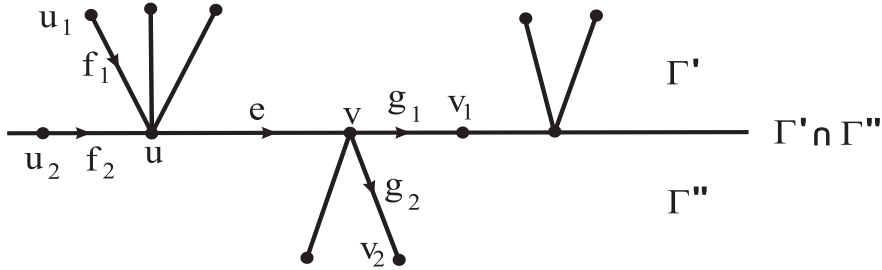}}
\end{center}
\caption{Union of mature graphs.}\label{example2}
\end{figure}

\begin{proof} From Proposition \ref{prop11} we know that the inclusions $(N_{\Gamma'}, \partial N_{\Gamma'})\to (N_{\Gamma}, \partial N_{\Gamma})$ and
$(N_{\Gamma''}, \partial N_{\Gamma''})\to (N_{\Gamma}, \partial N_{\Gamma})$ induce monomorphisms in two-dimensional homology. We want to show that the images of
the groups $H_2(N_{\Gamma'}, \partial N_{\Gamma'})$ and $H_2(N_{\Gamma''}, \partial N_{\Gamma''})$ generate $H_2(N_{\Gamma}, \partial N_{\Gamma})$.

Denote by $C$ the cellular chain complex $C_\ast(N_\Gamma, \partial N_{\Gamma})$. Similarly denote by $C'$ and $C''$ the cellular chain complexes
$C_\ast(N_{\Gamma'}, \partial N_{\Gamma'})$ and $C_\ast(N_{\Gamma''}, \partial N_{\Gamma''})$. An explicit description of $C, C', C''$ is given in \cite{BF}, \S 3.
The assumptions of the Proposition imply that $C=C'+C''$. Thus we have the exact sequence
$$\dots \to H_2(C')\oplus H_2(C'') \stackrel{f}{\to} H_2(C) \stackrel{\partial}\to H_1(C'\cap C'') \to \dots$$
Clearly the intersection $C'\cap C''$ coincides with the cellular chain complex $C_\ast(N_{\Gamma'\cap\Gamma''}, \partial N_{\Gamma'\cap\Gamma''})$.

Let us first deal with the case that $\Gamma' \cap \Gamma'' \ncong S^1$. The conditions of the proposition guarantee that $\Gamma' \cap \Gamma'' \ncong [0,1]$. This is because if $\Gamma' \cap \Gamma'' \cong [0,1]$ then the intersection has some extremal vertex $v$ and contains a single edge $e$ emanating from $v$. The other edges emanating from $v$ must all be contained in $\Gamma' - \Gamma''$ or all contained in $\Gamma'' - \Gamma'$. This means that $v$ is univalent in one of $\Gamma'$, $\Gamma''$ and hence one of them is not mature, by Proposition \ref{univalent}.

Now since we know $\Gamma' \cap \Gamma'' \ncong [0,1]$ and we are assuming $\Gamma' \cap \Gamma'' \ncong S^1$, we can use Corollary 2.5 from \cite{BF} to conclude that $H_1(C'\cap C'')=0$. Thus, the exact sequence above implies that the images of $H_2(C')$ and $H_2(C'')$ generate $H_2(C)$.

Consider the image of the intersection form $I_\Gamma: H_1(\Gamma) \otimes H_1(\Gamma) \to H_2(N_\Gamma, \partial N_\Gamma)=H_2(C)$. Since each of the graphs
$\Gamma'$ and $\Gamma''$ is mature, the images of the intersection forms $I_{\Gamma'}$ and $I_{\Gamma''}$ coincide with the subgroups $H_2(C')\subset H_2(C)$ and
$H_2(C'')\subset H_2(C)$. Thus, it follows that the intersection form $I_\Gamma$ is surjective.

Now let us look at the case that $\Gamma' \cap \Gamma'' \cong S^1$. In this case $$H_1(C'\cap C'') \cong H_1(N_{S^1}, \partial N_{S^1}) \cong H_1(S^1 \times S^1, F(S^1,2)) \cong \Z.$$ Thus $\mathrm{coker} \, f \cong \Z$. Since we know that $I_{\Gamma'}$ and $I_{\Gamma''}$ are surjective, it suffices to display an element $\alpha$ of $H_2(C)$ that generates $\mathrm{coker}(f)$ and show that it lies in $\mathrm{Im} (I_\Gamma)$.

Let $e$ be an edge of $\Gamma' \cap \Gamma''$ joining vertices $u$ and $v$ such that all edges emanating from $u$ lie in $\Gamma'$ and all edges emanating from $v$ lie in $\Gamma''$,
see Figure \ref{example2}. Let $f_i, g_i, u_i, v_i$ be as indicated in Figure \ref{example2} for $i=1,2$.
%\begin{figure}[h]
%\begin{center}
%\includegraphics{gamma8.eps}
%\end{center}
%\caption{The case $\Gamma' \cap \Gamma'' \cong S^1$.}\label{gamma8}
%\end{figure}

We construct our element $\alpha$ as follows. Since $\Gamma'$ is mature, $\Gamma' - \bar{e}$ is connected (by Proposition \ref{prop: non-mature}) so there is a path $\gamma_1$ in $\Gamma' - \bar{e}$ from $v_1$ to $u_1$. Similarly there is a path $\gamma_2$ in $\Gamma'' - \bar{e}$ from $v_2$ to $u_2$. Let $z_1$, $z_2$ be the cycles
\[
%\begin{array}{rcl}
z_1 =  \gamma_1 + f_1 + e + g_1, \quad
z_2  =  \gamma_2 + f_2 + e + g_2
%\end{array}
\]
and let $\alpha = I_\Gamma(z_1 \otimes z_2)$.

Next we show that $\alpha$ generates $\mathrm{coker}(f)$. Note that all elements $I_{\Gamma'}(z \otimes z')$ in $\mathrm{Im} (I_{\Gamma'})$ that contain $ee$ must contain $g_1g_1$ and vice versa. Similarly, all elements of $\mathrm{Im}(I_{\Gamma''})$ that contain $ee$ must contain $f_2f_2$ and vice versa. It follows that all elements of $\mathrm{Im}(f)$ that contain $k (ee)$ for some $k \in \Z$ must also contain $m (g_1 g_1) + n (f_2f_2)$ for some $m,n \in \Z$ with $m +n = k$. Note that $\alpha$ does not satisfy this property (as it has no terms $g_1g_1$ or $f_2f_2$), so $\alpha \notin \mathrm{Im}(f)$.

Suppose that $\alpha$ does not generate $\mathrm{coker}(f)$. Then there is some $l \geq 2$ and some $\beta \in H_2(C)$ such that $\alpha - l \beta \in \mathrm{Im}(f)$. The coefficient of $ee$ in $\alpha - l \beta$ is of the form $k = 1 + lk'$, the coefficient of $g_1g_1$ is of the form $m = lm'$ and the coefficient of $f_2f_2$ is of the form $n = ln'$. These can never satisfy $k = m + n$, so $\alpha - l \beta$ cannot be in $\mathrm{Im} (f)$.
\end{proof}

We will use Theorem \ref{thm0} to prove the following statements.
\begin{proposition}
The complete graph $K_n$ is mature for $n\ge 5$.
\end{proposition}
\begin{figure}[h]
\begin{center}
\resizebox{10cm}{5cm}{\includegraphics[50,501][558, 793]{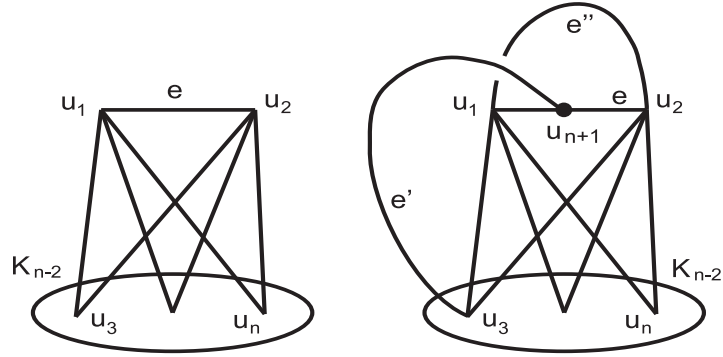}}
\end{center}
\caption{Modifying the complete graph.}\label{complete}
\end{figure}
\begin{proof} We use induction on $n$ starting with $n=5$; we know that $K_5$ is mature as shown in \cite{BF}.

Assuming that $K_n$ with $n\ge 5$ is mature,
we modify it by a sequence of moves ending at $K_{n+1}$, such that all the intermediate graphs we obtain are mature.

Let $u_1, \dots, u_n$ denote the vertices of $K_{n}$. Consider the edge $e$ connecting $u_1$ and $u_2$ (see Figure \ref{complete}, left). We subdivide it by introducing an
additional vertex $u_{n+1}$ in the middle. It is obvious that the new graph is also mature.

Next we add an edge $e'$ connecting $u_3$ and $u_{n+1}$, see Figure \ref{complete}, right. We claim that we may apply Theorem \ref{thm0} to conclude that the graph we obtain is mature. Indeed,
removing the vertices $u_3$ and $u_{n+1}$ from the graph obtained in the previous stage yields a connected graph; the graph we get deformation retracts onto $K_{n-1}$ with an edge removed.

Now we add a new edge $e''$ connecting $u_1$ and $u_2$. By Theorem \ref{thm0} the result is a mature graph.

Finally, we add (one by one) the edges connecting $u_4, \dots, u_n$ with $u_{n+1}$; in each of these cases Theorem \ref{thm0} is applicable.
 The graph obtained at the end is $K_{n+1}$.
\end{proof}

\begin{proposition}\label{prop21} The bipartite graph $K_{p,q}$ with $p\ge 3$ and $q\ge 3$ is mature.
\end{proposition}
\begin{proof} We assume that $K_{p,q}$ is mature for $p\ge 3, q\ge 3$ and prove that $K_{p,q+1}$ is mature; the result follows by induction as we know from \cite{BF} that $K_{3,3}$ is mature.

Let $V=P\sqcup Q$ be the set of vertices of $K_{p,q}$ where $|P|=p$, $|Q|=q$ and every vertex of $P$ is connected by an edge to every vertex of $Q$.

Add a new edge $e$ connecting two vertices $v_1, v_2 \in P$. The condition of Theorem \ref{thm0} is satisfied since the graph obtained by removing two vertices from $P$ deformation retracts onto
$K_{p-2, q}$ which is connected.

Next, subdivide $e$ by introducing a new vertex $q$ in the middle. The result is still mature.

Now we add an edge connecting $q$ to one of the remaining vertices $v_i\in P-\{v_1, v_2\}$. This edge addition satisfies Theorem \ref{thm0} and produces a mature graph.
This procedure may be repeated for every one of the vertices $v_i\in P-\{v_1, v_2\}$ and the final result is the graph $K_{p, q+1}$.
\end{proof}

\begin{corollary} 
\label{K_5 corollary}For the complete graph $\Gamma=K_n$ with $n\ge 5$ the configuration space $F(\Gamma, 2)$ has the following Betti numbers
$$b_1(F(\Gamma, 2))  = (n-1)(n-2),$$
$$b_2(F(\Gamma, 2)) = \frac{n(n-2)(n-3)(n-5)}{4} +1.$$
\end{corollary}
\begin{proof} Since we know that $K_n$ is mature we may apply formulae (\ref{dimone}) and (\ref{dimtwo}). The first Betti number of $\Gamma$ is the difference between the total number of edges and the number of edges in a spanning tree. Thus we have $b_1(\Gamma)= \binom{n}{2} - (n-1) = \binom {n-1} 2$ and $\mu(v)= n-1$ for every vertex $v$. Substituting into (\ref{dimtwo}) and making elementary transformations gives the indicated answer for the second Betti number.
\end{proof}

The expression for $b_2(F(\Gamma, 2))$ given by Corollary \ref{K_5 corollary} agrees with that found in \cite{CP}. 

Similarly we obtain:

\begin{corollary}
For the bipartite graph $\Gamma=K_{p,q}$ with $p\ge 3$ and $q\ge 3$ the configuration space $F(\Gamma, 2)$ has the following Betti numbers
$$b_1(F(\Gamma, 2)) = 2(p-1)(q-1),$$
$$b_2(F(\Gamma, 2)) = (p^2-3p+1)(q^2-3q+1).$$
\end{corollary}
\begin{proof} In this case we have $b_1(\Gamma)= (p-1)(q-1)$ and $\mu(v)=p$ (for vertices in $Q$) and $\mu(v)=q$ (for vertices in $P$).
Now one uses Proposition \ref{prop21} and formulae (\ref{dimone}) and (\ref{dimtwo}).
\end{proof}

\section{Further questions}

In this section we mention several open questions and conjectures.

1. We conjecture that a connected non-planar graph is mature if and only if it admits no decomposition as a wedge or double wedge.
This conjecture is inspired by Propositions \ref{maturity v connectedness} and \ref{double edge}
and the fact that we do not know of any non-mature graphs other than those covered in these propositions.

2. Consider a random graph $\Gamma\in G(n, p)$. Here $n$ is an integer, $0<p<1$, and $G(n,p)$ denotes the probability space of all subgraphs of the complete graph $K_n$
with each edge of $K_n$ included in $\Gamma$ with probability $p$, independently of all other edges.
This is the well-known Erd{\"{o}}s
- R\'{e}nyi model of random graphs. Note that the cardinality of $G(n,p)$ is $2^{\binom n 2}$
 and the probability that a specific graph $\Gamma$ appears as a result of a random process equals
\begin{eqnarray*}
{\mathbf P}(\Gamma)\,  = \, p^{E_\Gamma}(1-p)^{{\binom n 2}-E_\Gamma},
\end{eqnarray*}
where $E_\Gamma$ denotes the number of edges of $\Gamma$, see \cite{JLR}.

We believe that a random graph $\Gamma\in G(n, p)$
with parameter $p$ large enough is mature with high probability. One may want to find a threshold $0<p_c<1$ for maturity such that a random graph $\Gamma\in G(n, p)$ with $p>p_c$ is mature with probability tending to one as $n \to \infty$ and a random graph $\Gamma\in G(n, p)$
with $p<p_c$ is immature with probability tending to one as $n\to \infty$.

3. The concept of maturity may have interesting higher analogues relevant to the study of configuration spaces $F(\Gamma, k)$ with $k>2$, i.e. when one considers more than two points on a graph $\Gamma$.
The inclusion $$\alpha_k: F(\Gamma, k) \to \Gamma^k$$ (where $\Gamma^k$ denotes the Cartesian product of $k$ copies of $\Gamma$) induces an epimorphism
$$(\alpha_k)_\ast: H_1(F(\Gamma, k)) \to H_1(\Gamma^k)$$
(under very general assumptions on $\Gamma$, compare \cite{BF}, Proposition 1.3) and one says that a graph $\Gamma$ is $k$-mature if $(\alpha_k)_\ast$ is an isomorphism.
It would be useful to find examples and investigate properties of $k$-mature graphs for $k>2$.

4.  We do not know examples of graphs $\Gamma$ such that the homology group $H_1(F(\Gamma, 2))$ has nontrivial torsion and we conjecture that this group is always torsion free. This property is equivalent to the absence of torsion in the
 cokernel of the intersection form $I_\Gamma$ (see (\ref{intersection})), as follows from (\ref{six}).

%\vspace{\headheight}
%To be added:

%1. Explicit Betti numbers of $F(\Gamma, 2)$ for $\Gamma=K_n$ and $\Gamma=K_{p,q}$.

%2. Comparison with results of Copeland and Patty - their mistakes.

%3. Discussion of Theorems 11 and 13 (linking) and comparison with Copeland - Patty.

\bibliographystyle{amsalpha}

\end{document}